\documentclass{article}
\usepackage{lipsum} % Package to generate dummy text throughout this template
\usepackage{amsfonts}
\usepackage{amsmath}		
\usepackage{amssymb}
\usepackage{cite}
\usepackage{float}
\usepackage{graphicx}
\usepackage{booktabs, multicol, multirow}
\usepackage{bigstrut}
\usepackage{array}
\usepackage{tabularx}
\usepackage{pdflscape}
\usepackage{adjustbox}
\usepackage{longtable}
\usepackage{makeidx}
\usepackage{rotating}
\usepackage{wasysym}
\usepackage{graphicx}
\usepackage{graphics}
\usepackage{float}
\usepackage{latexsym}
\usepackage{amsfonts}
\usepackage{mathptmx}
\usepackage[ruled,lined]{algorithm2e}
\usepackage{authblk}
\usepackage{varwidth}
\usepackage{makeidx}
\usepackage{dsfont}
\usepackage[sc]{mathpazo} % Use the Palatino font
\usepackage[T1]{fontenc} % Use 8-bit encoding that has 256 glyphs
\usepackage{microtype} % Slightly tweak font spacing for aesthetics

\usepackage[hmarginratio=1:1,top=23mm,columnsep=20pt]{geometry} % Document margins
\usepackage{chngpage} %allows for temporary adjustment of side margins
\usepackage{hyperref} % For hyperlinks in the PDF

\usepackage[hang, small,labelfont=bf,up,textfont=it,up]{caption} % Custom captions under/above floats in tables or figures
\usepackage{booktabs} % Horizontal rules in tables
\usepackage{float} % Required for tables and figures in the multi-column environment - they need to be placed in specific locations with the [H] (e.g. \begin{table}[H])

\usepackage{lettrine} % The lettrine is the first enlarged letter at the beginning of the text
\usepackage{paralist} % Used for the compactitem environment which makes bullet points with less space between them

\usepackage{abstract} % Allows abstract customization
\makeatletter % remove the default symbols of \thanks
\let\@fnsymbol\normalsize
\makeatother
\usepackage{titlesec} % Allows customization of titles
\titleformat{\section}[block]{\large\scshape\centering}{\thesection.}{1em}{} % Change the look of the section titles
\usepackage{fancyhdr} % Headers and footers
\newtheorem{theorem}{Theorem}[section]
\newtheorem{lemma}[theorem]{Lemma}

\newtheorem{remark}[theorem]{Remark}
\newtheorem{assumption}[theorem]{Assumption}

\def\QED{~\rule[-1pt]{6pt}{6pt}\par\medskip}

\newenvironment{proof}{{\bf Proof\ }}{ \hfill \QED}
\date{}
\begin{document}
\thispagestyle{empty}
\pagenumbering{arabic}
\setlength {\baselineskip} {1 \baselineskip}

\title{\vspace{-15mm}\fontsize{18pt}{10pt}\selectfont\textbf{A diagonal PRP-type projection method for convex constrained nonlinear monotone equations}} % Article title

\author{{Hassan Mohammad$^{*}$ }
  \hfill    \\ 
 {\footnotesize {\it  $^{*}$ Department of Mathematical Sciences, Faculty of Physical Sciences, Bayero University, Kano.  Kano 700241, Nigeria.}}\\
 \hfill    \\ 
     {\footnotesize  Corresponding Author E-mail address: hmuhd.mth@buk.edu.ng }
      }

\maketitle % Insert title

%\thispagestyle{fancy} % All pages have headers and footers

%----------------------------------------------------------------------------------------
%	ABSTRACT
%----------------------------------------------------------------------------------------

\begin{abstract}
\noindent Iterative methods for nonlinear monotone equations do not require the differentiability assumption on the residual function. This special property of the methods makes them suitable for solving large-scale nonsmooth monotone equations. In this work, we present a diagonal Polak-Ribi$\grave{e}$re-Polyk (PRP) conjugate gradient-type method for solving large-scale nonlinear monotone equations with convex constraints. The search direction is a combine form  of a multivariate (\emph{diagonal}) spectral method and a modified PRP conjugate gradient method. Proper safeguards are devised to ensure positive definiteness of the diagonal matrix associated with the search direction. Based on Lipschitz continuity and monotonicity assumptions the method is shown to be globally convergent. Numerical results are presented by means of comparative experiments with recently proposed multivariate spectral conjugate gradient-type method.
\end{abstract}
\section{Introduction}
\label{intro}
In this paper, we address the following constrained monotone equations
\begin{equation}\label{eq1}
F(x^*)=0, ~~x^*\in \Omega, 
\end{equation}
where $\textbf{F} : \Omega \to \mathbb{R}^n$ is continuous and monotone function, on a non-empty closed convex set $\Omega \in \mathbb{R}^n$. The monotonicity of $F$ here means  $\langle\,F(x)-F(y),x-y\rangle \ge 0, ~~\forall x,y \in \Omega $

The system of monotone equations has various applications \cite{li2015s,gu2015,zheng2015image,eshaghnezhad2017}, e.g  the $\ell _1$-norm problem arising from compressing sensing \cite{xiao13,liu15,liu2017gradient}, generalized proximal algorithm with Bregman distances \cite{iusem1997}, variational inequalities problems \cite{qian2013,fukushima92}, and optimal power flow equations \cite{wood2012, ghaddar2016} among others.

At the case where $\Omega =\mathbb{R}^n$, Solodov and Svaiter \cite{solodov1998} combined Newton's method and projection strategy and proposed a global convergent inexact Newton method for system of monotone equations without the differentiability assumption. Wang et. al \cite{wang07}  extended the work by Solodov and Svaiter to solve convex constrained monotone equations. Since then many methods for solving system of nonlinear monotone equations have been proposed (see, e.g. \cite{yu09,wang09,lacruz14,sun2015,liu2014,wang2016,ou2018,ou2017,hu2015,liu2015,nayuan2017,zhang2011,mohdaba17,zhou2015prp,lacruz2017,li2011,cheng2009prp,yan2010,papp2015fr,abapoom18}, among others). Specifically, Ma and Wang \cite{ma2009} proposed a modified projection method for solving a system of monotone equations with convex constraints. Though the projection methods for convex constrained monotone equations proposed in \cite{wang07} and \cite{ma2009} have a very good numerical performance, they are not suitable for solving large-scale monotone equations because they require matrix storage.

By taking the advantage of the good performance of the multivariate (\emph{diagonal}) spectral gradient method by Han et. al. \cite{han2008}, Yu et. al. \cite{yu2013} proposed a multivariate spectral gradient projection method for nonlinear monotone equations with convex constraints. Recently, Liu and Li \cite{liu2017} proposed a multivariate spectral conjugate gradient-type projection algorithm for constrained nonlinear monotone equations by combining the multivariate spectral gradient method and the Dai-Yuan (DY) conjugate gradient method. 

Motivated by the development of the diagonal spectral gradient projection approach for constrained nonlinear monotone equations, in this paper, we extend the sufficient descent modified  Polak-Ribi$\grave{e}$re-Polyk (PRP) nonlinear conjugate gradient method proposed by Dong et. al.\cite{dong15} to solve nonlinear monotone equations with convex constraints. We show that the proposed method has some attractive properties, for example, it is a derivative-free and matrix-free method. Hence, it is suitable for solving non-smooth and large-scale problems. Under the monotonicity and Lipschitz continuity assumptions on $F$, we show that the proposed algorithm is globally convergent.

The outstanding part of this paper is organized as follows. In Section 2 we present the algorithm of the proposed method. In Section 3 we establish the global convergence of the proposed method. In section 4 we present the numerical experiments, and conclusions in Section 5. Unless otherwise stated, throughout this paper $\|.\|$ stands for the Euclidean norm of vectors and the induced $2$-norm of matrices,  $\langle\,.~,~.\rangle$ is the inner product of vectors in $\mathbb{R}^n$. \section{A diagonal PRP-type projection method}
We begin this section by describing the unconstrained convex optimization problem
%Consider the problem of finding the solution of the nonlinear system of equations of the form

\begin{equation}\label{eq2}
\min _{x\in \mathbb{R}^n}f(x),
\end{equation}
where $f:\mathbb{R}^n\to \mathbb{R}$ is a continuously differentiable function with gradient $g(x_k)=\nabla f(x_k)$. The descent modified Polak-Ribi$\grave{e}$re-Polyk (PRP) conjugate gradient method proposed by Dong et. al. \cite{dong15} generates a sequence $\{x_k\}$ that satisfies
\begin{equation}\label{eqscheme}
    x_{k+1}=x_k+\alpha _kd_k, ~~~k=0,1,...,
\end{equation}
where where $\alpha _k>0$ is a steplength, $d_k$ is a search direction generated by
\begin{equation}\label{eqdir1}
d_k=
    \begin{cases}
    -g(x_0), ~~~\text{if } k=0\\[10pt]
    -g(x_k)+\beta _kd_{k-1}, ~~~\text{if } k\geq 1,
    \end{cases}
\end{equation}
where $\beta _k$ is defined by
\begin{equation}\label{eqbita1}
    \beta _k=\max \biggl\{0, \frac{\langle g(x_k),y_{k-1}\rangle}{\|g(x_{k-1})\|^2}-t\frac{\langle g(x_k),d_{k-1}\rangle}{\|g(x_{k-1})\|^4}\biggl(\frac{\langle g(x_k),y_{k-1}\rangle}{\|g(x_{k})\|}\biggr)^2\biggr\},
\end{equation}
where $y_{k-1}=g(x_k)-g(x_{k-1}),~~t>\frac{1}{4}$ is a given positive constant.

Based on the formulae (\ref{eqscheme})-(\ref{eqbita1}), we introduce our method for solving (\ref{eq1}). Inspired by
(\ref{eqdir1}) we redefine $d_k$ as
\begin{equation}\label{eqdir2}
d_k=
    \begin{cases}
    -D_0F(x_0), ~~~\text{if } k=0\\[10pt]
    -D_kF(x_k)+\beta _kd_{k-1}, ~~~\text{if } k\geq 1,
    \end{cases}
\end{equation}
where $D_0=I_n,~ D_k$ is a positive definite diagonal matrix given by 
\begin{equation}\label{eqdiag}
    D_k=\text{diag}\biggl(\frac{1}{\lambda _k^1},\frac{1}{\lambda _k^2},...,\frac{1}{\lambda _k^n}\biggr),
\end{equation}
\begin{equation}\label{eqlambda}
\lambda_{k}^i=
\begin{cases}
{y_{k-1}^i\over s_{k-1}^i}, & \text{ if } {y_{k-1}^i\over s_{k-1}^i}> 0\\[15pt]
\delta,  & \text{ if }{y_{k-1}^i\over s_{k-1}^i}\le  0 \text{ or } s_{k-1}^i=0.
\end{cases}
\end{equation}
where $y_{k-1}=F(x_k)-F(x_{k-1}),~s_{k-1}=x_k-x_{k-1}, \text{ and }\delta >0$ is chosen as a safeguard that ensures all the entries of the diagonal matrix are positive. $\beta _k$ is defined as 
\begin{equation}\label{eqbita2}
    \beta _k=\max \biggl\{0, \frac{\langle F(x_k),y_{k-1}\rangle}{\|F(x_{k-1})\|^2}-t\frac{\langle F(x_k),d_{k-1}\rangle}{\|F(x_{k-1})\|^4}\biggl(\frac{\langle F(x_k),y_{k-1}\rangle}{\|F(x_{k})\|}\biggr)^2\biggr\}.
\end{equation}

Observe that, if the diagonal components the matrix $D_k$ are negative, a non-negative scalar $\delta$ is used as a safeguard. The common choice of $\delta$ in most of the previous works \cite{han2008,liu2017,yu2013} is \begin{equation}\label{eqsafe01}
\delta = \frac{\langle s_{k-1},~y_{k-1}\rangle}{\langle s_{k-1},~s_{k-1}\rangle}.
\end{equation} 
However, we feel that there may be better choices for replacing $\lambda _k^i$ when it is negative. In what follows, we present a simple approach for choosing the value of the scalar $\delta$.

The scalar $\lambda _{k}^i$  is non-positive if and only if 
%\begin{equation} \label{negcuv}
\begin{align*}
    & s_{k-1}^i<0 \text{ and } y_{k-1}^i\geq 0, \\[-5.5pt]
    \text{or}\\[-5.5pt] 
    & s_{k-1}^i>0 \text{ and } y_{k-1}^i\leq 0. 
\end{align*}
%\end{equation}

In the first case, it is better to choose a non-positive scalar that carries some information about the current value of $y_{k-1}$. Similarly, in the second case, it is better to choose a non-negative scalar that carries some information about the current value of $y_{k-1}$. We now present our simple safeguard for choosing the replacement of $\lambda _{k}^i$ in case it is negative. We consider the following two cases, in which $\theta \in (0,1)$ is a shrinking factor, and $\epsilon >0$ is a tolerance for ensuring strictly positive values, $i=1,2,..,n.$\\[-5pt]

\textbf{\underline{CASE I}}: Assume that $s_{k-1}^i >0$.
%\vspace{-3mm}
\[ \quad \text{If } y_{k-1}^i\leq 0 \implies F(x_{k})^i-F(x_{k-1})^i\leq 0 \implies F(x_{k})^i\leq F(x_{k-1})^i.\]

In this case, we redefine
\begin{equation}\label{safea1}
  y_{k-1}^i=\theta \max \bigl\{\max \bigl\{ | F(x_{k})^i|,|F(x_{k-1})^i|\bigr\}, \epsilon\bigr\},
\end{equation}
so that $\lambda_{k}^i={y_{k-1}^i\over s_{k-1}^i}> 0.$\\[-5pt]

\textbf{\underline{CASE II}}: Assume that $s_{k-1}^i <0.$
%\vspace{-3mm}
\[ \quad \text{If } y_{k-1}^i\geq 0 \implies F(x_{k})^i-F(x_{k-1})^i\geq 0 \implies F(x_{k})^i\geq F(x_{k-1})^i.\]

In this case, we redefine
\begin{equation}\label{safea2}
 y_{k-1}^i=-\theta \max \bigl\{\max \bigl\{ | F(x_{k})^i|,|F(x_{k-1})^i|\bigr\}, \epsilon\bigr\},
\end{equation}
so that $\lambda_{k}^i={y_{k-1}^i\over s_{k-1}^i}> 0.$

Next we state our algorithm. We used the projection operator  \[P_{\Omega}[x]:=\text{arg min}_{y\in \Omega} \{\|y-x\|\},\] that denotes the projection of $x$ on the closed convex set $\Omega$, to ensure the new point generated by our algorithm is a member of the domain of definition. It is not difficult to notice that if $x\in \Omega$, the value of $P_{\Omega}[x]$ is equal to $x$. In addition, $P_{\Omega}[x]$ is nonexpansive, that is
\begin{equation}\label{eqne}
    \|P_{\Omega}[x]-P_{\Omega}[y]\|\leq \|x-y\|, ~~~ \forall x,y\in \mathbb{R}^n.
\end{equation}
\vspace{-0.7cm}
\IncMargin{0.5em}
\begin{algorithm}
\caption{{Diagonal PRP Projection Method (DPPM)}}\label{algo_dpm}
\SetKwData{Left}{left}\SetKwData{This}{this}\SetKwData{Up}{up}
\SetKwFunction{Union}{Union}\SetKwFunction{FindCompress}{FindCompress}
\SetKwInOut{Input}{input}\SetKwInOut{Output}{output}
\nl\Input{ Given $x_0 \in \mathbb{R}^n,~ \Omega \in \mathbb{R}^n,~ \rho ,~\sigma ,~\theta, \in (0,1)$, ~$0<\ell \le 1\le u$  and $\epsilon,~\mu ,~ tol >0$.}
%\nl \Output{ $x_{k+1},~~k,~~\|F(x_{k+1})\|.$}
\BlankLine
\nl Set $k=0,~ D _k=I$, where $I$ is an $n\times n$ identity matrix\;
\nl Compute $\|F(x_k)\|$\;
\nl\While{$\|F(x_k)\|> tol$ }{
 \nl Compute $d_k=-D_kF(x_k)$\;\label{ln5}
 \nl Initialize $m =0$\;
 \nl \While{$\langle\,F(x_k+\beta \rho ^m d_k),~d_k\rangle >-\sigma \beta \rho ^m \|F(x_k+\beta \rho ^md_k)\|\|d_k\|^2$ \label{ln7}}{
\nl $m=m+1$\;
\nl $\alpha = \beta \rho ^m$\;
}
\nl Set $\alpha _k=\alpha$\;
\nl Compute $z_k=x_k+\alpha _kd_k$\;

\nl \If{$z_k\in \Omega \text{ and }F(z_k)\leq tol$ }{$x_{k+1}=z_k$}
\nl $x_{k+1}=P_{\Omega}[x_k-\xi _k F(x_k)],$ where
 \(\xi _k =\frac{\langle\, x_k-z_k,~F(z_k)\rangle}{\|F(z_k)\|^2}\)\;
 \nl Compute $s_k=x_{k+1}-x_k, ~~ y_k=F(x_{k+1})-F(x_{k})$\;
 \nl Compute \label{ln15}
 \begin{equation}\label{sdm3}
\lambda_{k+1}^i=
\begin{cases}
\max \bigl\{\min \bigl\{ \frac{y_{k}^i}{s_{k}^i},~~ u\bigr\},~~\ell \big\},  & \text{ if }s_{k}^i\ne 0 \\
1, &  \text{ if }s_{k}^i= 0.
\end{cases}
\end{equation} 
where $y_{k}^i$ is safeguarded for possible different signs with $s_k^i,$ using Equations (\ref{safea1}) and (\ref{safea2})\;
\nl Compute $D_{k+1}$ by (\ref{eqdiag}) and $\beta _{k+1}$ by  (\ref{eqbita2})\;\label{ln16}
\nl Compute \label{ln17}
\begin{equation}\label{eqnewdir}
d_{k+1}=
\begin{cases}
-D_{k+1}F(x_{k+1}), & \text{ if  } |\langle F(x_{k+1}),y_k\rangle|\|d_{k}\|\geq \mu \|F(x_{k+1})\|,\\[10pt]
-D_{k+1}F(x_{k+1})+\beta _{k+1} d_{k}, & \text{ if  } |\langle F(x_{k+1}),y_k\rangle|\|d_{k}\|< \mu \|F(x_{k+1})\|.
\end{cases}
\end{equation}

\nl Set $k=k+1$.
}
\end{algorithm}\DecMargin{1em}
\vspace{-0.7cm}
\begin{remark}
\item Since the matrix $D_k$ is diagonal, the product at line \ref{ln5} of Algorithm \ref{algo_dpm} is simply the  product between the elements of $D_k$ and the corresponding components of $F(x_k)$, computed in $\emph{O}(n)$ operations. Also by definition of $D_k$ at line \ref{ln16}, $D_k$ is positive definite.
\end{remark}
\begin{remark}
\item The sequence $\{\lambda _{k+1} ^i\}$ at line \ref{ln15}  is uniformly bounded for each $i\text{ and }k$. It is clear that \[\ell \leq \lambda _{k+1}^i\leq u~~\forall i, ~~\forall k.\]
\end{remark}
\vspace{-0.7cm}
\begin{remark}
\item The update rule for the search direction at line \ref{ln17}, is given specifically to ensure global convergence of the Algorithm DPPM.
\end{remark}
\vspace{-0.1cm}
Now, we show a nice property of the search direction defined in (\ref{eqnewdir}), which is quite significance in the  construction of our algorithm. The proof is similar to \cite{dong15} [Lemma 2.1] but with different definition for the right-hand term of the search direction.
\begin{lemma}\label{lm1}
Let $\{d_k\}$ be the sequence generated by (\ref{eqnewdir}), $\beta _k$ is given by (\ref{eqbita2}). Then there exists a constant $c>0$ such that
\begin{equation}\label{eqsd}
   \bigl\langle\,F(x_k),~d_k\bigr\rangle \leq -c\bigl\|F(x_k)\bigr\|^2,~~\forall k\ge 0 .
\end{equation}
\end{lemma}
\begin{proof}
For $k=0$, we have 
\begin{equation*}
    \begin{split}
        \langle F(x_0),~d_0\rangle &=\langle F(x_0),~-D_0F(x_0)\rangle\\
        &=\langle F(x_0),~-F(x_0)\rangle \quad \text{since }D_0=I \\
        &=-\|F(x_0)\|^2.
    \end{split}
\end{equation*}
For $k\geq 1,$ we consider two cases:

\textbf{Case (i)} If \(\frac{\langle F(x_k),y_{k-1}\rangle}{\|F(x_{k-1})\|^2}\leq t\frac{\langle F(x_k),d_{k-1}\rangle}{\|F(x_{k-1})\|^4}\biggl(\frac{\langle F(x_k),y_{k-1}\rangle}{\|F(x_{k})\|}\biggr)^2 \), then \(\beta _k=0\). Therefore,
\begin{equation*}
    \begin{split}
        \langle F(x_k),~d_k\rangle &=\langle F(x_k),~-D_kF(x_k)\rangle\\
        &=\biggl\langle F(x_k),~\text{diag}\biggl(\frac{-1}{\lambda _k^1},\frac{-1}{\lambda _k^2},...,\frac{-1}{\lambda _k^n}\biggr)F(x_k)\biggr\rangle \\
        & \leq \biggl\langle F(x_k),~\text{diag}\biggl(\frac{-1}{u },\frac{-1}{u},...,\frac{-1}{u}\biggr)F(x_k)\biggr\rangle \\
        & =- \frac{1}{u }\langle F(x_k),~F(x_k)\rangle\\
        &=-c\|F(x_0)\|^2, \quad 0<c=\frac{1}{u},
    \end{split}
\end{equation*}
where the first inequality follows from the boundedness of $\frac{-1}{\lambda _k^i}, ~~~i=1,2,\ldots n.$ 

\textbf{Case (ii)} If \(\frac{\langle F(x_k),y_{k-1}\rangle}{\|F(x_{k-1})\|^2}> t\frac{\langle F(x_k),d_{k-1}\rangle}{\|F(x_{k-1})\|^4}\biggl(\frac{\langle F(x_k),y_{k-1}\rangle}{\|F(x_{k})\|}\biggr)^2 \), then we have
\begin{equation*}
    \begin{split}
        \langle F(x_k),~d_k\rangle &=\langle F(x_k),~-D_kF(x_k)+\beta _kd_{k-1}\rangle\\
        &=\biggl\langle F(x_k),~\text{diag}\biggl(\frac{-1}{\lambda _k^1},\frac{-1}{\lambda _k^2},...,\frac{-1}{\lambda _k^n}\biggr)F(x_k)\biggr\rangle\\
        & \quad \quad \quad +\biggl\langle F(x_k),~\frac{\langle F(x_k),y_{k-1}\rangle}{\|F(x_{k-1})\|^2}d_{k-1}-t\frac{\langle F(x_k),d_{k-1}\rangle}{\|F(x_{k-1})\|^4}\biggl(\frac{\langle F(x_k),y_{k-1}\rangle}{\|F(x_{k})\|}\biggr)^2d_{k-1}\biggr\rangle \\
        & \leq \biggl\langle F(x_k),~\text{diag}\biggl(\frac{-1}{u },\frac{-1}{u},...,\frac{-1}{u}\biggr)F(x_k)\biggr\rangle \\
        &\quad \quad\quad +\biggl\langle F(x_k),~\frac{\langle F(x_k),y_{k-1}\rangle}{\|F(x_{k-1})\|^2}d_{k-1}\biggr\rangle+\biggl\langle F(x_k),~-t\frac{\langle F(x_k),d_{k-1}\rangle}{\|F(x_{k-1})\|^4}\biggl(\frac{\langle F(x_k),y_{k-1}\rangle}{\|F(x_{k})\|}\biggr)^2d_{k-1}\biggr\rangle
           \end{split}
\end{equation*}
\begin{equation}\label{eqcase2}
    \begin{split}
        &\quad \quad\quad\quad\quad =- \biggl(\frac{1}{u}-\frac{1}{4t}\biggr)\biggl\langle F(x_k),~F(x_k)\biggr\rangle +\biggl[-\frac{\bigl\langle F(x_k),~F(x_k)\bigr\rangle}{4t}+ \frac{\langle F(x_k),y_{k-1}\rangle}{\|F(x_{k-1})\|^2}\biggl\langle F(x_k),~d_{k-1}\biggr\rangle \\
        &\quad \quad\quad\quad\quad\quad \quad \quad-t\frac{\langle F(x_k),d_{k-1}\rangle^2}{\|F(x_{k-1})\|^4}\biggl(\frac{\langle F(x_k),y_{k-1}\rangle}{\|F(x_{k})\|}\biggr)^2\biggr]. 
    \end{split}
\end{equation}
By applying the inequality $2ab\leq a^2+b^2$ to the second term of (\ref{eqcase2}) with $a=\frac{1}{2\sqrt{t}}\|F(x_k)\|$ and $b=\sqrt{t}\frac{\langle F(x_k),d_{k-1}\rangle}{\|F(x_{k-1})\|^2}~.~\frac{\langle F(x_k),y_{k-1}\rangle}{\|F(x_{k})\|},$ we have 
 \[ \bigl\langle\,F(x_k),~d_k\bigr\rangle \leq -c\bigl\|F(x_k)\bigr\|^2,\]
 where $0<c=\frac{1}{u}-\frac{1}{4t} \implies t>\frac{u}{4},\quad u\ge 1$. 
 \end{proof}
\section{Convergence Analysis}
We now turn to analyze the global convergence of Algorithm \ref{algo_dpm} (DPPM). First we state the following standard assumptions:
\begin{assumption}\label{ass1}
The function $F:\mathbb{R}^n\to \mathbb{R}^n$ is continuously differentiable and monotone.
\end{assumption}
\begin{assumption}
The function $F$ is Lipschitz continuous on $\Omega$, i.e., there exists a constant $L>0$ such that
\begin{equation}\label{eqlc}
    \|F(x)-F(y)\|\leq L\|x-y\|, ~~~\forall x,y\in \Omega.
\end{equation}
\end{assumption}
\begin{assumption}\label{ass2}
The solution set of (\ref{eq1}) denoted by $\Phi$ is convex and nonempty.
\end{assumption}
In addition, we also assume that $F(x)\ne 0, \quad \forall x\ne x^*$, otherwise the solution of Equation (\ref{eq1}) has been obtained.
\begin{lemma}\label{lm2}
The Algorithm \ref{algo_dpm} (DPPM) is well defined.
\end{lemma}
\begin{proof}
Suppose that there exists $k_0\geq 0,$ such that \[\langle\,F(x_{k_0}+\beta \rho ^m d_{k_0}),~d_{k_0}\rangle >-\sigma \beta \rho ^m \|F(x_{k_0}+\beta \rho ^md_{k_0})\|\|d_{k_0}\|^2,~~\forall m\geq 0.\]
By continuity of $F$, allowing $m\to +\infty$, we have
\[\langle\,F(x_{k_0}),~d_{k_0}\rangle \geq 0,\] which contradits the conclusion of Lemma \ref{lm1}.
\end{proof}
The proof of the following Lemma is similar to that of \cite{liu15}.
\begin{lemma}\label{lm3}
Let $\{x_k\}, \text{ and }\{z_k\}$ be the sequences generated by Algorithm \ref{algo_dpm}. Then $\{x_k\}, \text{ and }\{z_k\}$ are bounded. In addition, 
\begin{align*}
   & \lim _{k\to +\infty}\|x_k-z_k\|=0, \text{ more specifically }\lim _{k\to +\infty}\alpha _k\|d_k\|=0.\\
   \text{and}\\
   & \lim _{k\to +\infty}\|x_{k+1}-x_k\|=0. \
\end{align*}
\end{lemma}
\begin{lemma}\label{lm4}
Let $\{x_k\}, \text{ and }\{d_k\}$ be the sequences generated by Algorithm \ref{algo_dpm}. If there exists a constant $\varepsilon >0$ such that $\|F(x_{k-1})\|\geq \varepsilon , \quad k\geq 1,$ then there exists a constant $\hat {c}>0$ such that
\begin{equation}\label{ls01}
    \|d_k\|\leq \hat {c}\|F(x_k)\|\quad \forall k\in \mathbb{N}\cup \{0\}.
\end{equation}
\end{lemma}
\begin{proof}
By the definition of $d_k$ in line \ref{ln17} of Algorithm \ref{algo_dpm}. If $|\langle F(x_{k}),y_{k-1}\rangle|\|d_{k-1}\|\geq \mu \|F(x_{k})\|,$ we have
\begin{align*}
    \|d_k\| &=\|D_kF(x_k)\|\\
    &\leq \|D_k\|\|F(x_k)\|\\
    &\leq \frac{\sqrt{n}}{\ell}\|F(x_k)\|\\
    &= \hat {c}\|F(x_k)\|,
\end{align*}
with $0<\hat{c}=\frac{\sqrt{n}}{\ell},$ where the first inequality follows from the induced matrix norm property, and the second inequality follows from the boundedness of the components of $D_k$.

On the other hand, if $|\langle F(x_{k}),y_{k-1}\rangle|\|d_{k-1}\|< \mu \|F(x_{k})\|,$ we have
\begin{align*}
    \|d_k\| &=\|-D_kF(x_k)+\beta _kd_{k-1}\|\\
    &\leq \|D_kF(x_k)\|+|\beta _k\|\|d_{k-1}\|\\
    &\leq  \|D_k\|\|F(x_k)\|+\frac{|\langle F(x_k),y_{k-1}\rangle|}{\|F(x_{k-1})\|^2}\|d_{k-1}\|\\
    & \quad \quad +\frac{|\langle F(x_k),d_{k-1}\rangle|}{\|F(x_{k-1})\|^4}~.~\frac{\langle F(x_k),y_{k-1}\rangle^2}{\|F(x_{k})\|^2}\|d_{k-1}\|\\
    &\leq \frac{\sqrt{n}}{\ell}\|F(x_k)\|+\frac{\mu \|F(x_k)\|}{\|F(x_{k-1})\|^2}+t\frac{\mu ^2\|F(x_k)\|}{\|F(x_{k-1})\|^4}\\
    &\leq \biggl(\frac{\sqrt{n}}{\ell}+\frac{\mu}{\varepsilon ^2}+t\frac{\mu ^2}{\varepsilon ^4}\biggr)\|F(x_k)\|\\
    &=\hat {c}\|F(x_k)\|,
\end{align*}
with $0<\hat{c}=\frac{\sqrt{n}}{\ell}+\frac{\mu}{\varepsilon ^2}+t\frac{\mu ^2}{\varepsilon ^4},$ where the second inequality is obtained from the induced matrix norm property, the third inequality follows directly from the boundedness of $D_k$, the Cauchy Schwartz inequality and the fact that $|\langle F(x_{k}),y_{k-1}\rangle|\|d_{k-1}\|< \mu \|F(x_{k})\|.$ And the proof is complete.
\end{proof}
 Now, we are ready to establish the global convergence of Algorithm \ref{algo_dpm}.
 \begin{theorem}
 Let $\{x_k\}, \text{ and }\{z_k\}$ be the sequences generated by Algorithm \ref{algo_dpm}. Then
 \begin{equation}\label{eqconv}
     \liminf_{k\to \infty}\|F(x_k)\|=0.
 \end{equation}
 Furthermore, the sequence $\{x_k\}$ converges to $\tilde{x}\in \Omega$ such that $F(\tilde{x})=0$.
 \end{theorem}
 \begin{proof}
 If $\displaystyle\liminf_{k\to \infty}\|d_k\|=0,$ we have  $\displaystyle\liminf _{k\to \infty}\|F(x_k)\|=0.$ 
 
Suppose $\displaystyle\liminf_{k\to \infty}\|d_k\|>0,$ we have  $\displaystyle\liminf _{k\to \infty}\|F(x_k)\|>0.$ \\
Using the line search in line (\ref{ln7}) of Algorithm \ref{algo_dpm}, \[-F(x_k+\beta \rho ^{m_{k-1}}d_k)^Td_k <\sigma \beta \rho ^{m_{k-1}}\|F(x_k+\beta\rho ^{m_{k-1}}d_k\|\|d_k\|^2\] 

and the boundedness of $\{x_k\}, \{d_k\}$ from Lemmas \ref{lm3}-\ref{lm4},  we can choose a subsequence such that allowing $k$ to go to infinity in the above inequality results
\begin{equation}\label{eqcntra1}
-\langle F(\tilde{x}),~\tilde{d}\rangle \leq 0.
\end{equation}
On the other hand, from Lemma (\ref{lm1}) we have
\begin{equation}\label{eqcontra2}
-\langle F(\tilde{x}),~\tilde{d}\rangle > 0.
\end{equation}
 Thus, (\ref{eqcntra1}) and (\ref{eqcontra2}) implies a contradiction. Therefore, (\ref{eqconv}) holds.

Furthermore, by continuity of $F$, the sequence $\{x_k\}$ has some accumulation point $\tilde{x}$ such that $F(\tilde{x})=0.$ Since $\{\|x_k-\tilde{x}\|\}$ converges (Lemma \ref{lm3}) and $\tilde{x}$ is an accumulation point of $\{x_k\}$, it follows that $\{x_k\}$ converges to $\tilde{x}$.\\

 \end{proof}
\section{Numerical Experiments}
In this section we present some numerical experiments to assess the efficiency of the proposed DPPM algorithm. We chose the following parameters for the implementation of DPPM algorithm  $\rho =0.8, \sigma =0.01, \theta = 0.1 ,\ell =\varepsilon =10^{-10}$, and $u=\mu = 10^{10}$.
In order to test the effectiveness of the proposed method, we compare it with the recent method (MDYP) proposed in \cite{liu2017}. The parameters in MDYP algorithm are chosen as in \cite{liu2017}. The initial trial stepsize $\beta$ is chosen uniformly for both algorithms as
\[ \beta =\frac{\langle F(x_k),~d_k\rangle}{\langle d_k, ~ (F(x_k+\gamma d_k)-F(x_k)\rangle /\gamma},\]
where $\gamma =10^{-8}$. If $\beta \le 10^{-6}$ holds, we set $\beta =1$.

We examine $5$ test problems (see, Appendix) using $5$ different dimensions $1000, 5000,10000$, $50000\text{ and }100000$ and $8$ different initial points (see, Table \ref{t1}).

The numerical results are reported in Tables \ref{t2}-\ref{t5} for number of iterations (ITER), number of function evaluation (FVAL), CPU time (TIME) and the norm of the residual function $F$ at the approximate solution (NORM). We chose the stopping criterion as follows:
\begin{itemize}
    \item $\|F(x_k)\|\le 10^{-5}$, and
    \item The number of iterations exceeds $1000$.
\end{itemize} 
A failure is reported and denoted by $\textbf{f}$ if any of the tested algorithms failed to satisfy the above stopping criterion.
All methods were coded in MATLAB R2017a and run on a Dell PC with an Intel Corei3 processor, 2.30GHz CPU speed, 4GB of RAM and Windows 7 operating system.

\begin{table}
\centering
  \caption{The initial points used for the test problems}
  %\vspace{\fill}
  \resizebox{6cm}{!}{%
\begin{tabular}{cl}
    \multicolumn{1}{c}{INITIAL POINT} & \multicolumn{1}{c}{VALUE} \\
    \midrule
       $x_1$   & $(1,1,\ldots ,1)^T$ \\
        $x_2$  & $(0.1,0.1,\ldots ,0.1)^T$ \\
        $x_3$  & $\bigl(\frac{1}{2}, \frac{1}{2^2},\ldots ,\frac{1}{2^n}\bigr)^T$ \\
       $x_4$   & $\bigl(1-\frac{1}{n}, 2-\frac{2}{n},\ldots ,n-1\bigr)^T$ \\
       $x_5$   & $\bigl(0,\frac{1}{n}, \ldots ,\frac{n-1}{n}\bigr)^T$ \\
       $x_6$   &  $\bigl(1,\frac{1}{2},\ldots ,\frac{1}{n}\bigr)^T$\\
       $x_7$   & $\bigl(\frac{n-1}{n}, \frac{n-2}{n},\ldots ,0 \bigr)^T$ \\
       $x_8$   &  $\bigl(\frac{1}{n}, \frac{2}{n},\ldots ,1\bigr)^T$\\
    \bottomrule
    \end{tabular}%
    }
  \label{t1}%
\end{table}%
\begin{figure}
%\vspace{-1.1cm}
\vspace{\fill}
\begin{center}
\includegraphics[width=15cm]{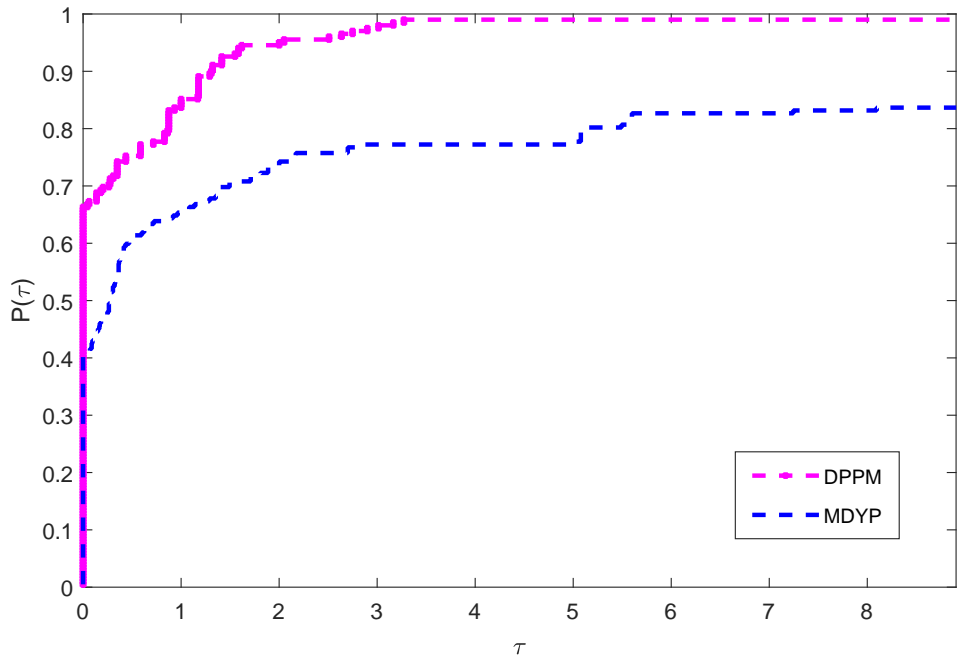}
\vspace{0.00001cm}{\textbf{\caption{Performance profile with respect to number of iterations (ITER)}\label{fig1}}} \addtocontents{lof}{\protect\addvspace{2cm}}
\end{center}

%\vspace{1.1cm}
\end{figure}

\begin{figure}
%\vspace{-5.5mm}
\vspace{\fill}
\begin{center}
\includegraphics[width=15cm]{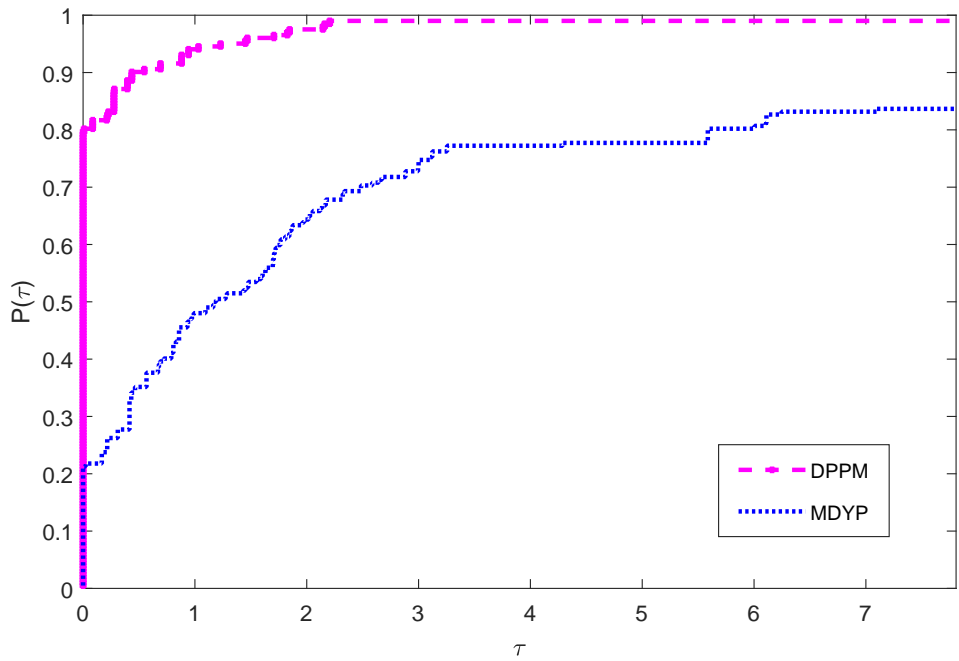}
\vspace{0.00001cm}{\textbf{\caption{Performance profile  with respect to number of function evaluations}\label{fig2}}} \addtocontents{lof}{\protect\addvspace{2cm}}
\end{center}
%\vspace{1.1cm}
\end{figure}

\begin{figure}
%\vspace{-1.1cm}
\vspace{\fill}
\begin{center}
\includegraphics[width=15cm]{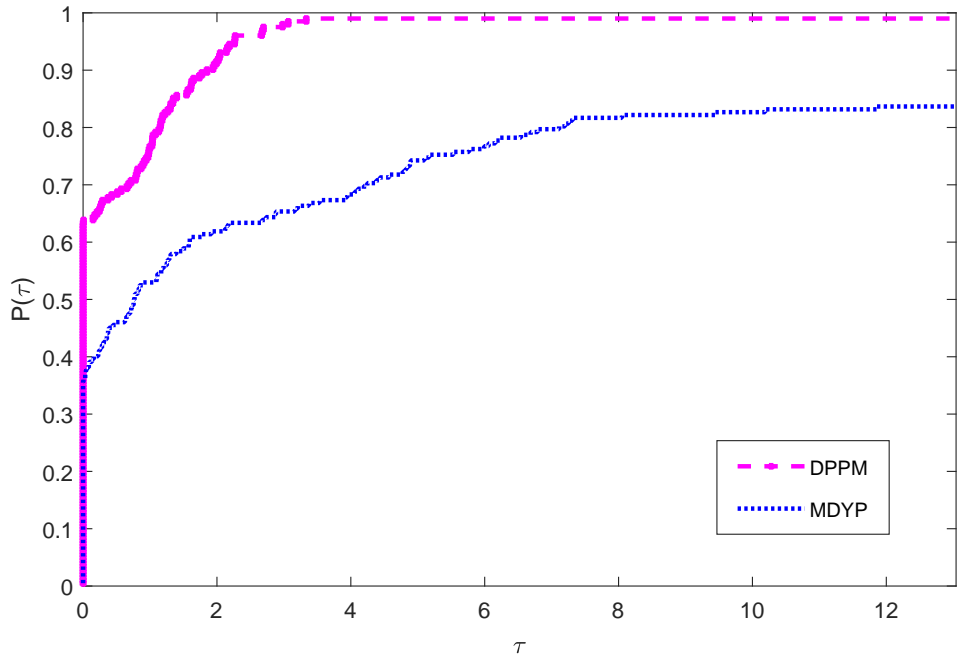}
\vspace{0.00001cm}{\textbf{\caption{Performance profile with respect to CPU time}\label{fig3}}} \addtocontents{lof}{\protect\addvspace{2cm}}
\end{center}
%\vspace{1.1cm}
\end{figure}

\begin{table}[htbp]
  \centering
  \caption{Numerical Results for DPPM and MDYP for Problem 1 with given initial points and dimension}
  \vspace{\fill}
  \resizebox{\columnwidth}{!}{%
    \begin{tabular}{cc|cccc|cccc}
    \toprule
          &       & \multicolumn{4}{|c|}{\textbf{DPPM}} & \multicolumn{4}{c}{\textbf{MDYP}} \\
    \midrule
    DIM   & INITIAL      POINT & ITER  & FVAL  & TIME  & NORM  & ITER  & FVAL  & TIME  & NORM \\
    \multirow{8}[16]{*}{1000} & $x_1$    & 38    & 82    & 0.39369 & 7.67E-06 & 33    & 161   & 0.11807 & 9.44E-06 \\
          & $x_2$    & 32    & 70    & 0.10322 & 7.15E-06 & 39    & 211   & 0.10552 & 9.40E-06 \\
          & $x_3$    & 120   & 252   & 0.32442 & 9.26E-06 & 49    & 256   & 0.70059 & 8.20E-06 \\
          & $x_4$    & 1     & 3     & 0.01260 & 0.00E+00 & 1     & 4     & 0.02221 & 0 \\
          & $x_5$    & 38    & 82    & 0.12030 & 8.15E-06 & 49    & 273   & 0.06990 & 8.50E-06 \\
          & $x_6$    & 37    & 79    & 0.04453 & 9.49E-06 & 57    & 311   & 0.06868 & 7.94E-06 \\
          & $x_7$    & 39    & 84    & 0.09649 & 8.57E-06 & 55    & 286   & 0.06015 & 7.37E-06 \\
          & $x_8$    & 37    & 80    & 0.11587 & 9.69E-06 & 56    & 291   & 0.09784 & 9.87E-06 \\
          \midrule
    \multirow{8}[16]{*}{5000} & $x_1$    & 39    & 84    & 0.20890 & 8.86E-06 & 52    & 260   & 0.32704 & 7.10E-06 \\
          & $x_2$    & 32    & 69    & 0.35442 & 8.07E-06 & 32    & 186   & 0.16735 & 8.25E-06 \\
          & $x_3$    & 120   & 252   & 0.63706 & 9.26E-06 & 53    & 283   & 4.06230 & 9.09E-06 \\
          & $x_4$    & 1     & 3     & 0.02699 & 0.00E+00 & 1     & 4     & 0.19802 & 0 \\
          & $x_5$    & 39    & 84    & 0.33551 & 9.77E-06 & 54    & 308   & 0.34295 & 7.27E-06 \\
          & $x_6$    & 37    & 79    & 0.23014 & 9.48E-06 & 61    & 328   & 0.49383 & 9.07E-06 \\
          & $x_7$    & 41    & 88    & 0.27738 & 8.21E-06 & 46    & 245   & 0.32441 & 9.16E-06 \\
          & $x_8$   & 39    & 84    & 0.26216 & 9.78E-06 & 61    & 327   & 0.30915 & 8.09E-06 \\
          \midrule
    \multirow{8}[16]{*}{10000} & $x_1$    & 40    & 86    & 0.36125 & 8.26E-06 & 42    & 208   & 0.30183 & 7.25E-06 \\
          & $x_2$   & 32    & 69    & 0.22630 & 9.73E-06 & 38    & 192   & 0.24845 & 7.01E-06 \\
          & $x_3$    & 120   & 252   & 1.21740 & 9.26E-06 & 53    & 283   & 10.89370 & 9.09E-06 \\
          & $x_4$    & 1     & 3     & 0.05100 & 0.00E+00 & 1     & 4     & 0.15397 & 0 \\
          & $x_5$    & 40    & 86    & 0.50858 & 9.29E-06 & 64    & 351   & 0.24912 & 8.09E-06 \\
          & $x_6$    & 37    & 79    & 0.30249 & 9.48E-06 & 42    & 216   & 0.79410 & 9.07E-06 \\
          & $x_7$    & 42    & 90    & 0.40105 & 7.71E-06 & 59    & 323   & 0.47751 & 8.01E-06 \\
          & $x_8$   & 40    & 86    & 0.34861 & 9.35E-06 & 78    & 434   & 0.59400 & 8.21E-06 \\
          \midrule
    \multirow{8}[16]{*}{50000} & $x_1$    & 120   & 250   & 4.32320 & 9.72E-06 & 41    & 206   & 1.63900 & 6.58E-06 \\
          & $x_2$    & 34    & 73    & 0.77121 & 7.41E-06 & 33    & 164   & 1.38310 & 7.51E-06 \\
          & $x_3$    & 120   & 252   & 1.88700 & 9.26E-06 & 53    & 283   & 134.87960 & 9.09E-06 \\
          & $x_4$   & 1     & 3     & 0.26091 & 0.00E+00 & 1     & 4     & 2.40870 & 0 \\
          & $x_5$    & 55    & 117   & 2.36960 & 8.44E-06 & 69    & 386   & 3.03180 & 7.28E-06 \\
          & $x_6$    & 37    & 79    & 1.52390 & 9.49E-06 & 49    & 256   & 1.06890 & 9.47E-06 \\
          & $x_7$    & 55    & 117   & 1.79620 & 8.60E-06 & 66    & 398   & 3.07370 & 6.51E-06 \\
          & $x_8$    & 55    & 117   & 1.38480 & 8.44E-06 & 61    & 356   & 1.54750 & 7.43E-06 \\
          \midrule
    \multirow{8}[16]{*}{100000} & $x_1$    & 156   & 323   & 6.91310 & 9.59E-06 & 63    & 321   & 2.36500 & 7.88E-06 \\
          & $x_2$    & 34    & 73    & 2.56550 & 8.93E-06 & 30    & 145   & 1.09880 & 8.96E-06 \\
          & $x_3$    & 120   & 252   & 4.15440 & 9.26E-06 & 53    & 283   & 460.82220 & 9.09E-06 \\
          & $x_4$    & 1     & 3     & 0.16653 & 0.00E+00 & 1     & 4     & 11.10340 & 0 \\
          & $x_5$   & 73    & 154   & 3.33150 & 9.20E-06 & 78    & 503   & 4.35460 & 9.93E-06 \\
          & $x_6$    & 37    & 79    & 1.37560 & 9.49E-06 & 50    & 259   & 3.27380 & 8.34E-06 \\
          & $x_7$    & 72    & 152   & 4.76270 & 9.88E-06 & 69    & 445   & 6.17050 & 9.63E-06 \\
          & $x_8$    & 73    & 154   & 3.36260 & 9.20E-06 & 59    & 345   & 3.37380 & 7.33E-06 \\
    \bottomrule
    \end{tabular}%
    }
  \label{t2}%
\end{table}%

\begin{table}[htbp]
  \centering
  \caption{Numerical Results for DPPM and MDYP for Problem 2 with given initial points and dimension}
  \vspace{\fill}
  \resizebox{\columnwidth}{!}{%
    \begin{tabular}{cc|cccc|cccc}
    \toprule
          &       & \multicolumn{4}{|c|}{\textbf{DPPM}} & \multicolumn{4}{c}{\textbf{MDYP}} \\
    \midrule
    DIMENSION & INITIAL      POINT & ITER  & FVAL  & TIME  & NORM  & ITER  & FVAL  & TIME  & NORM \\
    \multirow{8}[16]{*}{1000} & $x_1$    & 2     & 5     & 0.01205 & 0.00E+00 & 13    & 40    & 0.036007 & 3.39E-06 \\
          & $x_2$   & 2     & 5     & 0.00846 & 0.00E+00 & 7     & 22    & 0.019472 & 5.13E-06 \\
          & $x_3$    & 12    & 25    & 0.091209 & 4.33E-06 & 8     & 25    & 0.080773 & 6.13E-06 \\
          & $x_4$    & 4     & 20    & 0.051613 & 0.00E+00 & 130   & 391   & 0.12599 & 6.48E-06 \\
          & $x_5$   & 9     & 19    & 0.026844 & 0.00E+00 & 6     & 19    & 0.016339 & 0.00E+00 \\
          & $x_6$    & 9     & 19    & 0.023649 & 9.24E-07 & 4     & 13    & 0.012304 & 0 \\
          & $x_7$    & 9     & 19    & 0.027893 & 0.00E+00 & 6     & 19    & 0.013595 & 0 \\
          & $x_8$    & 9     & 19    & 0.023562 & 0     & 6     & 19    & 0.011906 & 0.00E+00 \\
          \midrule
    \multirow{8}[16]{*}{5000} & $x_1$   & 2     & 5     & 0.066537 & 0.00E+00 & 13    & 40    & 0.11026 & 7.50E-06 \\
          & $x_2$    & 2     & 5     & 0.022989 & 0.00E+00 & 8     & 25    & 0.068553 & 4.94E-06 \\
          & $x_3$    & 11    & 23    & 0.083563 & 9.62E-06 & 10    & 31    & 1.2681 & 3.78E-06 \\
          & $x_4$    & 3     & 18    & 0.04319 & 0.00E+00 & 455   & 1366  & 1.22  & 0 \\
          & $x_5$    & 11    & 23    & 0.15631 & 0.00E+00 & 6     & 19    & 0.01868 & 0.00E+00 \\
          & $x_6$    & 10    & 21    & 0.15303 & 2.04E-06 & 4     & 13    & 0.034344 & 0 \\
          & $x_7$    & 11    & 23    & 0.10842 & 0.00E+00 & 6     & 19    & 0.052581 & 0 \\
          & $x_8$   & 11    & 23    & 0.13773 & 0     & 6     & 19    & 0.045172 & 0.00E+00 \\
          \midrule
    \multirow{8}[16]{*}{10000} & $x_1$    & 2     & 5     & 0.067238 & 0.00E+00 & 14    & 43    & 0.11068 & 2.89E-06 \\
          & $x_2$    & 2     & 5     & 0.064798 & 0.00E+00 & 8     & 25    & 0.088129 & 6.92E-06 \\
          & $x_3$    & 12    & 25    & 0.15601 & 4.61E-06 & 10    & 31    & 3.2523 & 5.11E-06 \\
          & $x_4$    & 3     & 18    & 0.13358 & 0.00E+00 & 821   & 2464  & 4.5746 & 8.32E-06 \\
          & $x_5$    & 11    & 23    & 0.20715 & 5.65E-06 & 6     & 19    & 0.13156 & 0.00E+00 \\
          & $x_6$   & 10    & 21    & 0.15698 & 6.69E-06 & 4     & 13    & 0.06196 & 0 \\
          & $x_7$   & 11    & 23    & 0.13313 & 5.65E-06 & 6     & 19    & 0.064923 & 0 \\
          & $x_8$   & 11    & 23    & 0.12422 & 5.65E-06 & 6     & 19    & 0.068141 & 0.00E+00 \\
          \midrule
    \multirow{8}[16]{*}{50000} & $x_1$    & 23    & 50    & 1.2087 & 6.40E-06 & 14    & 43    & 0.60565 & 6.46E-06 \\
          & $x_2$    & 2     & 5     & 0.16224 & 0.00E+00 & 9     & 28    & 0.38936 & 1.05E-06 \\
          & $x_3$    & 11    & 23    & 0.51367 & 6.39E-06 & 10    & 31    & 27.959 & 5.37E-06 \\
          & $x_4$    & 3     & 18    & 0.36926 & 0.00E+00 &  f     &   f    &  f     & f \\
          & $x_5$   & 12    & 25    & 0.39204 & 6.62E-06 & 6     & 19    & 0.11452 & 0.00E+00 \\
          & $x_6$    & 12    & 25    & 0.74105 & 5.28E-07 & 4     & 13    & 0.073491 & 0 \\
          & $x_7$    & 12    & 25    & 0.50712 & 6.62E-06 & 6     & 19    & 0.11134 & 0 \\
          & $x_8$    & 12    & 25    & 0.59486 & 6.62E-06 & 6     & 19    & 0.2658 & 0.00E+00 \\
          \midrule
    \multirow{8}[16]{*}{100000} & $x_1$    & 56    & 119   & 4.2791 & 7.39E-06 & 14    & 43    & 1.0548 & 9.13E-06 \\
          &$x_2$   & 2     & 5     & 0.35045 & 0.00E+00 & 9     & 28    & 0.56004 & 1.49E-06 \\
          & $x_3$    & 11    & 23    & 0.8678 & 8.25E-06 & 10    & 31    & 89.3792 & 5.41E-06 \\
          & $x_4$    & 3     & 18    & 0.65425 & 0.00E+00 &   f    &   f    &  f     & f  \\
          & $x_5$   & 16    & 35    & 1.1296 & 9.42E-06 & 6     & 19    & 0.64241 & 0.00E+00 \\
          & $x_6$    & 12    & 25    & 0.75194 & 1.15E-06 & 4     & 13    & 0.38111 & 0 \\
          & $x_7$   & 16    & 35    & 1.5137 & 9.42E-06 & 6     & 19    & 0.48211 & 0 \\
          & $x_8$    & 16    & 35    & 1.0041 & 9.42E-06 & 6     & 19    & 0.57489 & 0.00E+00 \\
    \bottomrule
    \end{tabular}%
    }
  \label{t3}%
\end{table}%

\begin{table}[htbp]
  \centering
  \caption{Numerical Results for DPPM and MDYP for Problem 3 with given initial points and dimension}
  \vspace{\fill}
  \resizebox{\columnwidth}{!}{%
    \begin{tabular}{cc|cccc|cccc}
    \toprule
          &       & \multicolumn{4}{|c|}{\textbf{DPPM}} & \multicolumn{4}{c}{\textbf{MDYP}} \\
    \midrule
    DIMENSION & INITIAL      POINT & ITER  & FVAL  & TIME  & NORM  & ITER  & FVAL  & TIME  & NORM \\
    \multirow{8}[16]{*}{1000} & $x_1$ & 9     & 20    & 0.012702 & 9.41E-06 & 12    & 39    & 0.028509 & 9.82E-06 \\
          & $x_2$ & 8     & 18    & 0.027765 & 8.11E-06 & 10    & 32    & 0.012253 & 9.57E-06 \\
          & $x_3$ & 7     & 16    & 0.03306 & 8.25E-06 & 9     & 29    & 0.081886 & 2.35E-06 \\
          & $x_4$ & 242   & 496   & 0.43655 & 9.42E-06 & 36    & 152   & 0.090671 & 4.38E-06 \\
          & $x_5$ & 10    & 22    & 0.018539 & 4.98E-06 & 37    & 210   & 0.070034 & 2.98E-06 \\
          & $x_6$ & 9     & 20    & 0.016934 & 3.23E-06 & 17    & 62    & 0.020589 & 6.28E-06 \\
          & $x_7$ & 10    & 22    & 0.016338 & 4.98E-06 & 37    & 210   & 0.072627 & 2.98E-06 \\
          & $x_8$ & 10    & 22    & 0.015586 & 5.00E-06 & 22    & 99    & 0.034992 & 2.15E-06 \\
          \midrule
    \multirow{8}[16]{*}{5000} & $x_1$ & 10    & 22    & 0.10112 & 4.21E-06 & 13    & 42    & 0.058415 & 3.54E-06 \\
          & $x_2$ & 9     & 20    & 0.10815 & 3.63E-06 & 11    & 35    & 0.048767 & 3.71E-06 \\
          & $x_3$ & 7     & 16    & 0.046493 & 8.25E-06 & 9     & 29    & 1.1076 & 2.35E-06 \\
          & $x_4$ & 269   & 550   & 1.7034 & 9.31E-06 & 30    & 123   & 0.21588 & 3.44E-06 \\
          & $x_5$ & 11    & 24    & 0.08236 & 2.23E-06 & 36    & 209   & 0.37061 & 3.49E-06 \\
          & $x_6$ & 9     & 20    & 0.072288 & 3.24E-06 & 22    & 87    & 0.13458 & 6.38E-06 \\
          & $x_7$ & 11    & 24    & 0.14758 & 2.23E-06 & 36    & 209   & 0.50368 & 3.49E-06 \\
          & $x_8$ & 11    & 24    & 0.081127 & 2.23E-06 & 20    & 84    & 0.13688 & 8.44E-06 \\
          \midrule
    \multirow{8}[16]{*}{10000} & $x_1$ & 10    & 22    & 0.14814 & 5.95E-06 & 13    & 42    & 0.12261 & 5.00E-06 \\
          & $x_2$ & 9     & 20    & 0.081035 & 5.13E-06 & 11    & 35    & 0.066221 & 5.24E-06 \\
          & $x_3$ & 7     & 16    & 0.067335 & 8.25E-06 & 9     & 29    & 3.1159 & 2.35E-06 \\
          & $x_4$ & 280   & 572   & 1.7613 & 9.79E-06 & 45    & 244   & 0.57493 & 7.67E-06 \\
          & $x_5$ & 11    & 24    & 0.053928 & 3.16E-06 & 31    & 177   & 0.39748 & 1.82E-06 \\
          & $x_6$ & 9     & 20    & 0.14158 & 3.24E-06 & 18    & 72    & 0.1053 & 8.19E-06 \\
          & $x_7$ & 11    & 24    & 0.16218 & 3.16E-06 & 31    & 177   & 0.27428 & 1.82E-06 \\
          & $x_8$ & 11    & 24    & 0.16983 & 3.16E-06 & 29    & 153   & 0.17573 & 6.17E-06 \\
          \midrule
    \multirow{8}[16]{*}{50000} & $x_1$ & 43    & 92    & 2.0725 & 9.57E-06 & 14    & 45    & 0.52068 & 4.52E-06 \\
          & $x_2$ & 10    & 22    & 0.33935 & 2.29E-06 & 12    & 38    & 0.36153 & 4.63E-06 \\
          & $x_3$ & 7     & 16    & 0.14291 & 8.25E-06 & 9     & 29    & 22.695 & 2.35E-06 \\
          & $x_4$ & 307   & 626   & 7.8832 & 9.69E-06 & 38    & 174   & 1.9062 & 9.53E-06 \\
          & $x_5$ & 17    & 37    & 0.36148 & 4.12E-06 & 44    & 295   & 9.8907 & 6.42E-06 \\
          & $x_6$ & 9     & 20    & 0.13877 & 3.24E-06 & 23    & 90    & 8.835 & 6.56E-06 \\
          & $x_7$ & 17    & 37    & 0.34472 & 4.12E-06 & 44    & 295   & 9.8494 & 6.41E-06 \\
          & $x_8$ & 17    & 37    & 0.79731 & 4.12E-06 & 36    & 222   & 2.2409 & 4.06E-06 \\
          \midrule
    \multirow{8}[16]{*}{100000} & $x_1$ & 58    & 123   & 4.0543 & 7.82E-06 & 14    & 45    & 0.63779 & 6.39E-06 \\
          & $x_2$ & 10    & 22    & 1.0512 & 3.24E-06 & 12    & 38    & 0.46115 & 6.55E-06 \\
          & $x_3$ & 7     & 16    & 0.59808 & 8.25E-06 & 9     & 29    & 87.9961 & 2.35E-06 \\
          & $x_4$ & 319   & 650   & 15.465 & 9.35E-06 & 33    & 147   & 3.2193 & 7.98E-06 \\
          & $x_5$ & 32    & 69    & 2.4051 & 9.29E-06 & 34    & 224   & 3.9158 & 4.61E-06 \\
          & $x_6$ & 9     & 20    & 0.70022 & 3.24E-06 & 21    & 81    & 10.7951 & 4.64E-06 \\
          & $x_7$ & 32    & 69    & 2.3197 & 9.29E-06 & 34    & 224   & 4.1301 & 4.61E-06 \\
          & $x_8$ & 32    & 69    & 2.5293 & 9.29E-06 & 37    & 225   & 2.6843 & 4.58E-06 \\
    \bottomrule

    \end{tabular}%
    }
  \label{t4}%
\end{table}%

\begin{table}[htbp]
  \centering
  \caption{Numerical Results for DPPM and MDYP for Problem 4 with given initial points and dimension}
  \vspace{\fill}
  \resizebox{\columnwidth}{!}{%
    \begin{tabular}{cc|cccc|cccc}
    \toprule
          &       & \multicolumn{4}{|c|}{\textbf{DPPM}} & \multicolumn{4}{c}{\textbf{MDYP}} \\
    \midrule
    DIMENSION & INITIAL      POINT & ITER  & FVAL  & TIME  & NORM  & ITER  & FVAL  & TIME  & NORM \\
    \multirow{8}[16]{*}{1000} & $x_1$ & 1     & 3     & 0.006994 & 0.00E+00 & f     & f     & f     & f \\
          & $x_2$ & 5     & 11    & 0.016786 & 0.00E+00 & f     & f     & f     & f \\
          & $x_3$ & 9     & 19    & 0.024407 & 1.87E-06 & 304   & 913   & 3.7303 & 9.98E-06 \\
          & $x_4$ & 185   & 382   & 1.0441 & 0.00E+00 & f     & f     & f     & f \\
          & $x_5$ & 11    & 23    & 0.058111 & 9.06E-06 & f     & f     & f     & f \\
          & $x_6$ & 19    & 40    & 0.074199 & 6.95E-06 & 856   & 2569  & 2.6683 & 1.00E-05 \\
          & $x_7$ & 11    & 23    & 0.042337 & 9.06E-06 & f     & f     & f     & f \\
          & $x_8$ & 24    & 50    & 0.14774 & 4.23E-06 & f     & f     & f     & f \\
          \midrule
    \multirow{8}[16]{*}{5000} & $x_1$ & 1     & 3     & 0.020524 & 0.00E+00 & f     & f     & f     & f \\
          & $x_2$ & 5     & 11    & 0.090725 & 0.00E+00 & f     & f     & f     & f \\
          & $x_3$ & 9     & 19    & 0.090678 & 1.87E-06 & 304   & 913   & 24.0402 & 9.98E-06 \\
          & $x_4$ & 203   & 418   & 2.2173 & 0.00E+00 & 870   & 2612  & 22.9668 & 9.99E-06 \\
          & $x_5$ & 13    & 27    & 0.26952 & 4.42E-06 & f     & f     & f     & f \\
          & $x_6$ & 19    & 40    & 0.32377 & 7.56E-06 & 918   & 2755  & 5.3977 & 9.98E-06 \\
          & $x_7$ & 13    & 27    & 0.35444 & 4.42E-06 & f     & f     & f     & f \\
          & $x_8$ & 25    & 52    & 0.42654 & 3.97E-06 & f     & f     & f     & f \\
          \midrule
    \multirow{8}[16]{*}{10000} & $x_1$ & 1     & 3     & 0.026783 & 0.00E+00 & f     & f     & f     & f \\
          & $x_2$ & 5     & 11    & 0.15274 & 0.00E+00 & f     & f     & f     & f \\
          & $x_3$ & 9     & 19    & 0.092308 & 1.87E-06 & 304   & 913   & 64.261 & 9.98E-06 \\
          & $x_4$ & 211   & 434   & 2.9705 & 0.00E+00 & f     & f     & f     & f \\
          & $x_5$ & 13    & 27    & 0.53653 & 2.93E-06 & f     & f     & f     & f \\
          & $x_6$ & 19    & 40    & 0.76099 & 7.49E-06 & 921   & 2764  & 8.8082 & 9.99E-06 \\
          & $x_7$ & 13    & 27    & 0.37658 & 2.93E-06 & f     & f     & f     & f \\
          & $x_8$ & 25    & 52    & 0.80643 & 5.62E-06 & f     & f     & f     & f \\
          \midrule
    \multirow{8}[16]{*}{50000} & $x_1$ & 1     & 3     & 0.13612 & 0.00E+00 & f     & f     & f     & f \\
          & $x_2$ & 5     & 11    & 0.75059 & 0.00E+00 & f     & f     & f     & f \\
          & $x_3$ & 9     & 19    & 0.64213 & 1.87E-06 & 304   & 913   & 750.0724 & 9.98E-06 \\
          & $x_4$ & 229   & 470   & 13.7584 & 0.00E+00 & f     & f     & f     & f \\
          & $x_5$ & 11    & 24    & 1.5439 & 2.13E-06 & f     & f     & f     & f \\
          & $x_6$ & 19    & 40    & 1.8662 & 7.45E-06 & 921   & 2764  & 35.4218 & 1.00E-05 \\
          & $x_7$ & 11    & 24    & 0.89235 & 2.13E-06 & f     & f     & f     & f \\
          & $x_8$ & 11    & 24    & 0.69469 & 2.12E-06 & f     & f     & f     & f \\
          \midrule
    \multirow{8}[16]{*}{100000} & $x_1$ & 1     & 3     & 0.17122 & 0.00E+00 & f     & f     & f     & f \\
          & $x_2$ & 5     & 11    & 1.2485 & 0.00E+00 & f     & f     & f     & f \\
          & $x_3$ & 9     & 19    & 0.73873 & 1.87E-06 & 304   & 913   & 2739.78 & 9.98E-06 \\
          & $x_4$ & 237   & 486   & 28.0708 & 0.00E+00 & f     & f     & f     & f \\
          & $x_5$ & 15    & 33    & 2.7825 & 1.35E-06 & f     & f     & f     & f \\
          & $x_6$ & 19    & 40    & 2.4742 & 7.45E-06 & 921   & 2764  & 73.568 & 1.00E-05 \\
          & $x_7$ & 15    & 33    & 3.0101 & 1.35E-06 & f     & f     & f     & f \\
          & $x_8$ & 15    & 33    & 1.8552 & 1.35E-06 & f     & f     & f     & f \\
    \bottomrule
    \end{tabular}%
    }
  \label{t5}%
\end{table}%

\begin{table}[htbp]
  \centering
  \caption{Numerical Results for DPPM and MDYP for Problem 4 with given initial points and dimension}
  \vspace{\fill}
  \resizebox{\columnwidth}{!}{%
    \begin{tabular}{cc|cccc|cccc}
    \toprule
          &       & \multicolumn{4}{|c|}{\textbf{DPPM}} & \multicolumn{4}{c}{\textbf{MDYP}} \\
    \midrule
    DIMENSION & INITIAL      POINT & ITER  & FVAL  & TIME  & NORM  & ITER  & FVAL  & TIME  & NORM \\
    \multirow{8}[16]{*}{1000} & $x_1$ & 20    & 44    & 0.019887 & 7.58E-06 & 11    & 38    & 0.025455 & 3.21E-06 \\
          & $x_2$ & 8     & 18    & 0.022029 & 8.00E-06 & 9     & 29    & 0.009151 & 8.71E-06 \\
          & $x_3$ & 12    & 27    & 0.028013 & 4.40E-06 & 12    & 40    & 0.12209 & 1.70E-06 \\
          & $x_4$ & 1     & 3     & 0.004682 & 0.00E+00 & 1     & 4     & 0.014134 & 0 \\
          & $x_5$ & 17    & 37    & 0.035034 & 5.83E-06 & 21    & 79    & 0.031607 & 6.66E-06 \\
          & $x_6$ & 14    & 31    & 0.017107 & 8.77E-06 & 11    & 36    & 0.02408 & 3.32E-06 \\
          & $x_7$ & 17    & 37    & 0.042426 & 5.83E-06 & 21    & 79    & 0.028564 & 6.66E-06 \\
          & $x_8$ & 17    & 37    & 0.051895 & 5.84E-06 & 17    & 61    & 0.023712 & 6.86E-06 \\
          \midrule
    \multirow{8}[16]{*}{5000} & $x_1$ & 21    & 46    & 0.18004 & 8.27E-06 & 11    & 38    & 0.040818 & 7.19E-06 \\
          & $x_2$ & 9     & 20    & 0.084557 & 3.58E-06 & 10    & 32    & 0.044893 & 8.33E-06 \\
          & $x_3$ & 12    & 27    & 0.099041 & 4.40E-06 & 12    & 40    & 1.7057 & 1.70E-06 \\
          & $x_4$ & 1     & 3     & 0.025493 & 0.00E+00 & 1     & 4     & 0.16941 & 0 \\
          & $x_5$ & 23    & 50    & 0.17716 & 5.38E-06 & 19    & 69    & 0.095303 & 2.22E-06 \\
          & $x_6$ & 14    & 31    & 0.22375 & 8.76E-06 & 11    & 36    & 0.058293 & 3.32E-06 \\
          & $x_7$ & 23    & 50    & 0.31409 & 5.38E-06 & 19    & 69    & 0.065043 & 2.22E-06 \\
          & $x_8$ & 23    & 50    & 0.14158 & 5.38E-06 & 27    & 122   & 0.17497 & 3.72E-06 \\
          \midrule
    \multirow{8}[16]{*}{10000} & $x_1$ & 22    & 48    & 0.24683 & 5.71E-06 & 12    & 41    & 0.21346 & 4.72E-06 \\
          & $x_2$ & 9     & 20    & 0.08746 & 5.06E-06 & 11    & 35    & 0.09006 & 1.49E-06 \\
          & $x_3$ & 12    & 27    & 0.051866 & 4.40E-06 & 12    & 40    & 3.8951 & 1.70E-06 \\
          & $x_4$ & 1     & 3     & 0.018948 & 0.00E+00 & 1     & 4     & 0.55884 & 0 \\
          & $x_5$ & 23    & 50    & 0.20947 & 7.61E-06 & 17    & 62    & 0.16185 & 9.13E-06 \\
          & $x_6$ & 14    & 31    & 0.19068 & 8.76E-06 & 11    & 36    & 0.061976 & 3.32E-06 \\
          & $x_7$ & 23    & 50    & 0.2562 & 7.61E-06 & 17    & 62    & 0.13079 & 9.13E-06 \\
          & $x_8$ & 23    & 50    & 0.31204 & 7.61E-06 & 29    & 114   & 0.66644 & 9.92E-06 \\
          \midrule
    \multirow{8}[16]{*}{50000} & $x_1$ & 74    & 156   & 2.2287 & 8.50E-06 & 13    & 44    & 0.34412 & 7.07E-07 \\
          & $x_2$ & 10    & 22    & 0.59981 & 2.26E-06 & 11    & 35    & 0.24221 & 3.33E-06 \\
          & $x_3$ & 12    & 27    & 0.30703 & 4.40E-06 & 12    & 40    & 28.6017 & 1.70E-06 \\
          & $x_4$ & 1     & 3     & 0.12101 & 0.00E+00 & 1     & 4     & 2.2182 & 0 \\
          & $x_5$ & 32    & 69    & 0.69063 & 7.53E-06 & 18    & 65    & 0.19135 & 8.62E-06 \\
          & $x_6$ & 14    & 31    & 0.30697 & 8.76E-06 & 11    & 36    & 0.38975 & 3.32E-06 \\
          & $x_7$ & 32    & 69    & 1.7399 & 7.53E-06 & 18    & 65    & 0.44778 & 8.62E-06 \\
          & $x_8$ & 32    & 69    & 0.96137 & 7.53E-06 & 18    & 65    & 0.37787 & 8.97E-06 \\
          \midrule
    \multirow{8}[16]{*}{100000} & $x_1$ & 97    & 203   & 4.8452 & 8.54E-06 & 13    & 44    & 0.75655 & 1.00E-06 \\
          & $x_2$ & 10    & 22    & 0.43113 & 3.20E-06 & 11    & 35    & 0.4291 & 4.71E-06 \\
          & $x_3$ & 12    & 27    & 0.83201 & 4.40E-06 & 12    & 40    & 115.0905 & 1.70E-06 \\
          & $x_4$ & 1     & 3     & 0.39874 & 0.00E+00 & 1     & 4     & 8.6205 & 0 \\
          & $x_5$ & 43    & 92    & 2.8254 & 8.61E-06 & 19    & 68    & 1.1444 & 2.35E-06 \\
          & $x_6$ & 14    & 31    & 1.1267 & 8.76E-06 & 11    & 36    & 0.38034 & 3.32E-06 \\
          & $x_7$ & 43    & 92    & 2.3561 & 8.61E-06 & 19    & 68    & 1.0552 & 2.35E-06 \\
          & $x_8$ & 43    & 92    & 2.8962 & 8.61E-06 & 19    & 68    & 0.70571 & 2.41E-06 \\
    \bottomrule
    \end{tabular}%
    }
  \label{t6}%
\end{table}%

Figures \ref{fig1}-\ref{fig3} shows the Dolan and Mor$\acute{e}$ \cite{dolan2002} performance profiles based on the number of iterations,  the number of functions evaluations, and CPU time. 

It can be observed from Figure \ref{fig1} , DPPM solves and wins  almost $70\%$, while MSDY solves and wins less than $50\%$ of the $200$ problems. In terms of the number of function evaluation Figure \ref{fig2} shows that DPPM outperforms MDYP, as it solves about $80\%$ of the problems with the least number of functions evaluation. This good performance of SSGM2 is related to the efficiency of its search direction since it requires less function evaluation for computing the accepted steplength. 

In terms of the CPU time metric, DPPM is faster than MDYP with more than $60\%$ success. In summary, the proposed DPPM method is more successful than MDYP method based on the numerical experiments presented. This is due to the improvement in the search direction.

\section{Conclusions and Future Research}
We proposed a diagonal Polak-Ribi$\grave{e}$re-Polyk (PRP) type projection method for solving convex constrained nonlinear monotone equations problems (DPPM). The proposed method requires neither the exact Jacobian computation nor its storage space. This is an advantage, especially for large-scale nonsmooth problems. The global convergence of the proposed method is obtained without any merit function or differentiability assumptions on the residual function. Moreover, we proposed a simple strategy for ensuring that the components of the diagonal matrices are safely positive at each iteration.

In future research, we intend to study better techniques to use for improving the search direction with a suitable line search strategy. Because of the low memory requirement of the DPPM method, we hope it will perform well when applied to solve practical problems arising from compressed sensing and image processing this is also another subject for future research.

\section*{Appendix}
In this section we list the test problems used for the numerical experiments.\\
\indent \textbf{Problem 1} \cite{lacruz2006} Modified Exponential Function
\begin{equation*}
\begin{aligned}
F_1(x)&=e^{x_1}-1, \\
F_i(x)&=e^{x_i}-x(i-1)-1, ~ \text{for }i=2,3,...,n.\\
\text{and } &\Omega=R_{+}^n.
\end{aligned}
\end{equation*}
%\\[-5.5pt]
\indent \textbf{Problem 2} \cite{lacruz2006} Modified Logarithmic Function
\begin{equation*}
\begin{aligned}
F_i(x)& =\ln(|x_i|+1)-\frac{x_i}{n}, ~ \text{for }i=2,3,...,n.\\
\text{and } &\Omega=R_{+}^n.
\end{aligned}
\end{equation*}
\indent \textbf{Problem 3} Nonsmooth Function \cite{zhouli07}
\begin{equation*}
\begin{aligned}
F_i(x)& =2x_i-\sin |x_i|, ~ i=1,2,3,...,n.\\
\text{and } &\Omega=R_{+}^n.
\end{aligned}
\end{equation*}
\indent \textbf{Problem 4}\cite{lacruz2017} \begin{equation*}
\begin{aligned}
F_i(x)&=\min(\min(|x_i|,x_{i}^2),\max(|x_i|,x_{i}^3)), ~ i=1,2,3,...,n.\\
\text{and } &\Omega=R_{+}^n.
\end{aligned}
\end{equation*}
\indent \textbf{Problem 5} \cite{lacruz2006} Strictly Convex Function
\begin{equation*}
\begin{aligned}
F_i(x)&=e^{x_i}-1, ~ \text{for }i=1,2,...,n.\\
\text{and } &\Omega=R_{+}^n.
\end{aligned}
\end{equation*}

\section*{Acknowledgements}
We would like to thank Professor (Associate) Jinkui Liu affiliated at the Chongqing Three Gorges University, Chongqing, China, for providing us with access to the MDYP MATLAB codes.

% We also wish to thank Hasan Shahid of  Brown University, USA for his careful reading and helpful comments.
\newpage
\bibliographystyle{acm}
\bibliography{dprpref}
\end{document}